\documentclass[a4paper,11pt,leqno]{smfart}

\usepackage{amssymb,amsmath,enumerate,verbatim,mathrsfs,graphics,graphicx,mathptm,float,mathtools,nccmath}
\usepackage[T1]{fontenc}
\usepackage[latin1]{inputenc}
\usepackage[english, francais]{babel}
\usepackage[all]{xy}
\usepackage[pdfstartview=FitH,
            colorlinks=true,
            linkcolor=blue,
            urlcolor=blue,
            citecolor=red,
            bookmarks,
            bookmarksopen=true,
            bookmarksnumbered=true]{hyperref}

\usepackage{cleveref}
\usepackage{xcolor}
\theoremstyle{plain}

\newtheorem{theoalph}{Theorem}

\newtheorem{thmalph}[theoalph]{Theorem}
\newtheorem{propalph}[theoalph]{Proposition}
\newtheorem{coralph}[theoalph]{Corollary}

\theoremstyle{definition}

\theoremstyle{remark}

\theoremstyle{plain}
\newtheorem{thmsec}{Theorem}[section]

\newtheorem*{theolaidmihoubi}{Theorem}
\newtheorem{pro}[thmsec]{Proposition}
\newtheorem{lem}[thmsec]{Lemma}

\theoremstyle{definition}

\theoremstyle{remark}

\def\og{\leavevmode\raise.3ex\hbox{$\scriptscriptstyle\langle\!\langle$~}}
\def\fg{\leavevmode\raise.3ex\hbox{~$\!\scriptscriptstyle\,\rangle\!\rangle$}}

\addtocounter{section}{0}             
\numberwithin{equation}{section}       

\begin{document}
\selectlanguage{english}
\title[Quadrinomial coefficients]{Congruences concerning quadrinomial coefficients}

\date{\today}

\author{Mohammed \textsc{Mechacha}}

\address{Facult\'e de Math\'ematiques, USTHB, BP $32$, El-Alia, $16111$ Bab-Ezzouar, Alger, Alg\'erie}
\email{mmechacha@usthb.dz}

\maketitle{}

\begin{abstract}
In this paper, we establish congruences (mod $p^2$) involving the quadrinomial coefficients $\dbinom{np-1}{p-1}_{3}$ and~$\dbinom{np-1}{\frac{p-1}{2}}_{3}.$ This is an analogue of congruences involving the trinomial coefficients $\dbinom{np-1}{p-1}_{2}$ and $\dbinom{np-1}{\frac{p-1}{2}}_{2}$ due to Elkhiri and Mihoubi.

\end{abstract}

\section{Introduction and statements of results}
\bigskip

\noindent Several mathematicians studied, for a prime number $p$, congruences modulo powers of $p$ involving the binomial coefficients $\dbinom{2p-1}{p-1}$ and $\dbinom{p-1}{\frac{p-1}{2}}$, see for instance~\cite{Bab19,Wol62,Mor95,Glai00Quart1,Car53Canad,Car53Lond,Zha07}.

\noindent In $2014$, \textsc{Sun} \cite{Sun14} and \textsc{Cao}~\&~\textsc{Pan} \cite{CP14} described some properties and congruences involving the trinomial coefficients $\dbinom{n}{k}_{2}$ defined by
\[
\left(1+x+x^{2}\right)^{n}=\overset{2n}{\underset{k=0}{\sum}}\binom{n}{k}_{2}x^{k}.
\]

\noindent In $2019$, \textsc{Elkhiri} and \textsc{Mihoubi} \cite{EM19arxiv} studied congruences modulo $p^{2}$ for the trinomial coefficients $\dbinom{np-1}{p-1}_{2}$ and $\dbinom{np-1}{\frac{p-1}{2}}_{2}.$ More precisely, they proved the following result.

\begin{theolaidmihoubi}[\cite{EM19arxiv}, \rm{Theorem~1}]
{\sl Let $p\geq5$ be a prime number and $n$ be a positive integer. We have
\begin{equation*}
\hspace{-1.8cm}\binom{np-1}{p-1}_{2}\equiv
\left\{
\begin{array}{ll}
\hspace{-1mm}1+npq_{p}(3)\hspace{1mm}(\mathrm{mod}\hspace{0.3mm} p^2) &\text{if}\hspace{1.5mm} p\equiv1\,(\mathrm{mod}\hspace{0.3mm} 3),
\medskip\\
\hspace{-1mm}-1-npq_{p}(3)\hspace{1mm}(\mathrm{mod}\hspace{0.3mm} p^2) &\text{if}\hspace{1.5mm} p\equiv2\,(\mathrm{mod}\hspace{0.3mm} 3),
\end{array}
\right.
\end{equation*}
and
\begin{equation*}\label{equa:congruence-Bi3nomial-np-1-p-1-sur-2}
\binom{np-1}{\frac{p-1}{2}}_{2}\equiv
\left\{
\begin{array}{ll}
\hspace{-1mm}1+np\Big(2q_{p}(2)+\frac{1}{2}q_{p}(3)\Big)\hspace{1mm}(\mathrm{mod}\hspace{0.3mm} p^2) &\text{if}\hspace{1.5mm} p\equiv1\,(\mathrm{mod}\hspace{0.3mm} 6),
\medskip\\
\hspace{-1mm}-\frac{1}{2}npq_{p}(3)\hspace{1mm}(\mathrm{mod}\hspace{0.3mm} p^2) &\text{if}\hspace{1.5mm} p\equiv5\,(\mathrm{mod}\hspace{0.3mm} 6),
\end{array}
\right.
\end{equation*}
where $q_{p}(a)$ is the \textsc{Fermat} quotient defined, for $a\in\mathbb{Z}-p\mathbb{Z}$, by $q_{p}(a)=\frac{a^{p-1}-1}{p}.$
}
\end{theolaidmihoubi}

\noindent Analogously, one defines the quadrinomial coefficients $\binom{n}{k}_{3}$ by
\[
\left(1+x+x^2+x^3\right)^{n}=\overset{3n}{\underset{k=0}{\sum}}\binom{n}{k}_{3}x^{k}.
\]
In this paper, we study congruences modulo $p^{2}$ for the quadrinomial coefficients $\dbinom{np-1}{p-1}_{3}$ and $\dbinom{np-1}{\frac{p-1}{2}}_{3}.$ Our~purpose is to establish the following analogous result.
\begin{thmalph}\label{thmalph:Bi3nomial-np-1-p-1-p-1-sur-2}
{\sl Let $p\geq5$ be a prime number and $n$ be a positive integer. We have
\begin{equation}\label{equa:congruence-Bi3nomial-np-1-p-1}
\hspace{-1.8cm}\binom{np-1}{p-1}_{3}\equiv
\left\{
\begin{array}{ll}
\hspace{-1mm}1+2npq_{p}(2)\hspace{1mm}(\mathrm{mod}\hspace{0.3mm} p^2) &\text{if}\hspace{1.5mm} p\equiv1\,(\mathrm{mod}\hspace{0.3mm} 4),
\medskip\\
\hspace{-1mm}-\frac{1}{2}npq_{p}(2)\hspace{1mm}(\mathrm{mod}\hspace{0.3mm} p^2) &\text{if}\hspace{1.5mm} p\equiv3\,(\mathrm{mod}\hspace{0.3mm} 4),
\end{array}
\right.
\end{equation}
and
\begin{equation}\label{equa:congruence-Bi3nomial-np-1-p-1-sur-2}
\binom{np-1}{\frac{p-1}{2}}_{3}\equiv
\left\{
\begin{array}{ll}
\hspace{-1mm}1+np\Big(\frac{13}{4}q_{p}(2)+\chi_{p}\Big)\hspace{1mm}(\mathrm{mod}\hspace{0.3mm} p^2) &\text{if}\hspace{1.5mm} p\equiv1\,(\mathrm{mod}\hspace{0.3mm} 8),
\medskip\\
\hspace{-1mm}-1-np\Big(\frac{13}{4}q_{p}(2)+\chi_{p}\Big)\hspace{1mm}(\mathrm{mod}\hspace{0.3mm} p^2) &\text{if}\hspace{1.5mm} p\equiv3\,(\mathrm{mod}\hspace{0.3mm} 8),
\medskip\\
\hspace{-1mm}-np\Big(\frac{1}{4}q_{p}(2)-\chi_{p}\Big)\hspace{1mm}(\mathrm{mod}\hspace{0.3mm} p^2) &\text{if}\hspace{1.5mm} p\equiv5\,(\mathrm{mod}\hspace{0.3mm} 8),
\medskip\\
\hspace{-1mm}np\Big(\frac{1}{4}q_{p}(2)-\chi_{p}\Big)\hspace{1mm}(\mathrm{mod}\hspace{0.3mm} p^2) &\text{if}\hspace{1.5mm} p\equiv7\,(\mathrm{mod}\hspace{0.3mm} 8),
\end{array}
\right.
\end{equation}
where $\chi_{p}:=P_{p-(\frac{2}{p})}/p$ is the \textsc{Pell} quotient and $(P_{n})_{n}$ is the Pell sequence (OEIS A000129).
}
\end{thmalph}

\noindent Using this theorem we obtain the following proposition, which is an analogue of Proposition~3~of~\cite{EM19arxiv}.
\begin{propalph}\label{propalph:sommes-Bi3nomial-np-1-k}
{\sl
Let $p\geq5$ be a prime number and $n$ be a positive integer. Then
\begin{equation}\label{equa:sommes-Bi3nomial-np-1-k-k-0-p-1}
\hspace{-1.4cm}\overset{p-1}{\underset{k=0}{\sum}}\binom{np-1}{k}_{3}\equiv \left\{
\begin{array}{ll}
\hspace{-1mm}1+\frac{9}{4}npq_{2}(2)\hspace{1mm}(\mathrm{mod}\hspace{0.3mm}p^{2})&\text{if}\hspace{1.5mm} p\equiv 1\,(\mathrm{mod}\hspace{0.3mm} 4),
\medskip\\
\hspace{-1mm}-\frac{1}{4}npq_{2}(2)\hspace{1mm}(\mathrm{mod}\hspace{0.3mm}p^{2})&\text{if}\hspace{1.5mm} p\equiv 3\,(\mathrm{mod}\hspace{0.3mm} 4),
\end{array}
\right.
\end{equation}
and
\begin{equation}\label{equa:sommes-Bi3nomial-np-1-k-k-0-p-1-sur-2}
\overset{\frac{p-1}{2}}{\underset{k=0}{\sum}}\binom{np-1}{k}_{3}\equiv \left\{
\begin{array}{ll}
\hspace{-1mm}1+\frac{3}{2}np\Big(2q_{p}(2)+\chi_{p}\Big)\hspace{1mm}(\mathrm{mod}\hspace{0.3mm}p^{2})&\text{if}\hspace{1.5mm} p\equiv 1\,(\mathrm{mod}\hspace{0.3mm} 8),
\medskip\\
\hspace{-1mm}-\frac{1}{4}np\Big(q_{p}(2)-2\chi_{p}\Big)\hspace{1mm}(\mathrm{mod}\hspace{0.3mm}p^{2})&\text{if}\hspace{1.5mm} p\equiv 3\,(\mathrm{mod}\hspace{0.3mm} 8),
\medskip\\
\hspace{-1mm}-\frac{1}{2}np\Big(q_{p}(2)-\chi_{p}\Big)\hspace{1mm}(\mathrm{mod}\hspace{0.3mm}p^{2})&\text{if}\hspace{1.5mm} p\equiv 5\,(\mathrm{mod}\hspace{0.3mm} 8),
\medskip\\
\hspace{-1mm}-\frac{1}{4}np\Big(q_{p}(2)+2\chi_{p}\Big)\hspace{1mm}(\mathrm{mod}\hspace{0.3mm}p^{2})&\text{if}\hspace{1.5mm} p\equiv 7\,(\mathrm{mod}\hspace{0.3mm} 8).
\end{array}
\right.
\end{equation}
}
\end{propalph}

\noindent Corollary~4~of~\cite{EM19arxiv} gives congruences modulo $p^{2}$ for the coefficients $\binom{np^{2}-1}{k}_{2}.$ A similar congruences for the coefficients $\binom{np^{2}-1}{k}_{3}$ is given by the following statement.
\begin{coralph}\label{coralph:Bi3nomial-np2-1-k}
{\sl
Let $p\geq5$ be a prime number and $n,k$ be integers with $n\geq1 $
and $k\in \left \{  0,1,\ldots,p-1\right \}  .$ We~have
\begin{equation}\label{equa:Bi3nomial-np2-1-k}
\binom{np^{2}-1}{k}_{3}\equiv \left \{
\begin{array}{ll}
\hspace{-1mm}1\hspace{1mm}(\mathrm{mod}\hspace{0.3mm}p^{2})&\text{if}\hspace{1.5mm} k\equiv 0\,(\mathrm{mod}\hspace{0.3mm} 4),
\medskip\\
\hspace{-1mm}-1\hspace{1mm}(\mathrm{mod}\hspace{0.3mm}p^{2})&\text{if}\hspace{1.5mm} k\equiv 1\,(\mathrm{mod}\hspace{0.3mm} 4),
\medskip\\
\hspace{-1mm}0\hspace{1mm}(\mathrm{mod}\hspace{0.3mm}p^{2})&\text{if}\hspace{1.5mm} k\equiv 2\,(\mathrm{mod}\hspace{0.3mm} 4),
\medskip\\
\hspace{-1mm}0\hspace{1mm}(\mathrm{mod}\hspace{0.3mm}p^{2})&\text{if}\hspace{1.5mm} k\equiv 3\,(\mathrm{mod}\hspace{0.3mm} 4).
\end{array}
\right.
\end{equation}

}
\end{coralph}

\section{Some intermediate results}
\bigskip

\noindent In this section, we give some congruences which will be useful for establishing our main results.
\begin{lem}[\rm{\emph{cf.}} \cite{Glai00Quart2,Wil91}]
{\sl Let $p$ be a prime number. We have
\begin{align}
H_{[ p/2]}& \equiv-2q_{p}(2)\hspace{1mm}(\mathrm{mod}\hspace{0.3mm}p),\,p\geq3,\label{Harmonic2}\\
H_{[p/4]} & \equiv-3q_{p}(2)\hspace{1mm}(\mathrm{mod}\hspace{0.3mm}p),\,p\geq5,\label{Harmonic4}\\
H_{[p/8]} & \equiv-4q_{p}(2)-2\chi_{p}\hspace{1mm}(\mathrm{mod}\hspace{0.3mm}p),\,p\geq5,\label{Harmonic8}
\end{align}
where $(H_{n})_n$ is the harmonic sequence defined by $H_{0}=0$ and $H_{n}=1+\frac{1}{2}+\cdots+\frac{1}{n}.$
}
\end{lem}

\noindent Congruences~(\ref{Harmonic2})~and~(\ref{Harmonic4}) are due to \textsc{Glaisher}~\cite[pages~21-23]{Glai00Quart2}. Congruence~(\ref{Harmonic8}) is derived from a paper of \textsc{Williams}~\cite[page~440]{Wil91}.


\begin{lem}
{\sl Let $p\geq5$ be a prime number.

\noindent\textbf{\textit{1.}} If $p\equiv 1\,(\mathrm{mod}\hspace{0.3mm} 4)$ then
\begin{small}
\begin{align}\label{equa:sum-p=1mod4}
&\hspace{3mm}\sum\limits_{k=0}^{\frac{p-5}{4}}\frac{1}{4k+1}\equiv\frac{3}{4}q_{p}(2)\,(\mathrm{mod}\hspace{0.3mm} p),&&
\sum\limits_{k=0}^{\frac{p-1}{4}}\frac{1}{4k+2}\equiv1-\frac{1}{4}q_{p}(2)\,(\mathrm{mod}\hspace{0.3mm} p),&&
\sum\limits_{k=0}^{\frac{p-1}{4}}\frac{1}{4k+3}\equiv\frac{1}{2}+\frac{1}{4}q_{p}(2)\,(\mathrm{mod}\hspace{0.3mm} p).
\end{align}
\end{small}

\noindent\textbf{\textit{2.}} If $p\equiv 3\,(\mathrm{mod}\hspace{0.3mm} 4)$ then
\begin{small}
\begin{align}\label{equa:sum-p=3mod4}
&\sum\limits_{k=0}^{\frac{p-3}{4}}\frac{1}{4k+1}\equiv\frac{1}{4}q_{p}(2)\,(\mathrm{mod}\hspace{0.3mm} p),&&
\sum\limits_{k=0}^{\frac{p-3}{4}}\frac{1}{4k+2}\equiv-\frac{1}{4}q_{p}(2)\,(\mathrm{mod}\hspace{0.3mm} p),&&
\sum\limits_{k=0}^{\frac{p-7}{4}}\frac{1}{4k+3}\equiv\frac{3}{4}q_{p}(2)\,(\mathrm{mod}\hspace{0.3mm} p).
\end{align}
\end{small}
}
\end{lem}


\begin{proof}


\textbf{\textit{i.}} Assume that $p\equiv 1\,(\mathrm{mod}\hspace{0.3mm} 4).$ We have
\begin{fleqn}[1.5cm]
\begin{flalign*}
\sum \limits_{k=0}^{ \frac{p-5}{4} }\frac{1}{4k+1}&=\sum\limits_{j=1}^{\frac{p-1}{4}}\frac{1}{4( \frac{p-1}{4}-j)+1}=\sum\limits_{j=1}^{\frac{p-1}{4}}\frac{1}{p-4j},\\
\sum \limits_{k=0}^{ \frac{p-1}{4} }\frac{1}{4k+2}&=\sum\limits_{j=\frac{p-1}{4}}^{\frac{p-1}{2}}\frac{1}{4( \frac{p-1}{2}-j)+2}=\sum\limits_{j=\frac{p-1}{4}}^{\frac{p-1}{2}}\frac{1}{2p-4j}
\end{flalign*}
\end{fleqn}
{\fontsize{11}{11pt}\text{and}}
\begin{fleqn}[1.5cm]
\begin{flalign*}
\sum \limits_{k=0}^{ \frac{p-1}{4} }\frac{1}{4k+3}&=\sum\limits_{j=\frac{p+3}{4}}^{\frac{p+1}{2}}\frac{1}{4( j- \frac{p+3}{4})+3}=\sum\limits_{j=\frac{p+3}{4}}^{\frac{p+1}{2}}\frac{1}{4j-p},
\end{flalign*}
\end{fleqn}
so that
\begin{fleqn}[1.5cm]
\begin{flalign*}
\sum \limits_{k=0}^{ \frac{p-5}{4} }\frac{1}{4k+1}&\equiv-\frac{1}{4}\sum \limits_{j=1}^{\frac{p-1}{4}}\frac{1}{j}
=-\frac{1}{4}H_{[\frac{p}{4}]}=\frac{3}{4}q_{p}(2)\,(\mathrm{mod}\hspace{0.3mm} p),\\
\sum \limits_{k=0}^{ \frac{p-1}{4} }\frac{1}{4k+2}&\equiv -\frac{1}{4}\sum \limits_{j=\frac{p-1}{4}}^{\frac{p-1}{2}}\frac{1}{j} =-\frac{1}{p-1}-\frac{1}{4}\sum\limits_{j=1}^{\frac{p-1}{2}}\frac{1}{j}+\frac{1}{4}\sum\limits_{j=1}^{\frac{p-1}{4}}\frac{1}{j}\\
&\equiv1-\frac{1}{4}\sum\limits_{j=1}^{\frac{p-1}{2}}\frac{1}{j}+\frac{1}{4}\sum\limits_{j=1}^{\frac{p-1}{4}}\frac{1}{j}=1-\frac{1}{4}H_{[\frac{p}{2}]}+\frac{1}{4}H_{[\frac{p}{4}]}=1-\frac{1}{4}q_{p}(2)\,(\mathrm{mod}\hspace{0.3mm} p)
\end{flalign*}
\end{fleqn}
{\fontsize{11}{11pt}\text{and}}
\begin{fleqn}[1.5cm]
\begin{flalign*}
\sum \limits_{k=0}^{ \frac{p-1}{4} }\frac{1}{4k+3}&\equiv\frac{1}{4}\sum \limits_{j=\frac{p+3}{4}}^{\frac{p+1}{2}}\frac{1}{j} =\frac{1}{2p+2}+\frac{1}{4}\sum\limits_{j=1}^{\frac{p-1}{2}}\frac{1}{j}-\frac{1}{4}\sum\limits_{j=1}^{\frac{p-1}{4}}\frac{1}{j}\\
&\equiv\frac{1}{2}+\frac{1}{4}\sum\limits_{j=1}^{\frac{p-1}{2}}\frac{1}{j}-\frac{1}{4}\sum\limits_{j=1}^{\frac{p-1}{4}}\frac{1}{j}=\frac{1}{2}+\frac{1}{4}H_{[\frac{p}{2}]}-\frac{1}{4}H_{[\frac{p}{4}]}=\frac{1}{2}+\frac{1}{4}q_{p}(2)\,(\mathrm{mod}\hspace{0.3mm} p).
\end{flalign*}
\end{fleqn}


\textbf{\textit{ii.}} Assume now that $p\equiv 3\,(\mathrm{mod}\hspace{0.3mm} 4).$ From the equalities
\begin{fleqn}[1.5cm]
\begin{flalign*}
\sum \limits_{k=0}^{ \frac{p-3}{4} }\frac{1}{4k+1}&=\sum\limits_{j=\frac{p+1}{4}}^{\frac{p-1}{2}}\frac{1}{4(j- \frac{p-3}{4}-1)+1}=\sum\limits_{j=\frac{p+1}{4}}^{\frac{p-1}{2}}\frac{1}{4j-p},\\
\sum \limits_{k=0}^{ \frac{p-3}{4} }\frac{1}{4k+2}&=\sum\limits_{j=\frac{p+1}{4}}^{\frac{p-1}{2}}\frac{1}{4( \frac{p-1}{2}-j)+2}=\sum\limits_{j=\frac{p+1}{4}}^{\frac{p-1}{2}}\frac{1}{2p-4j}
\end{flalign*}
\end{fleqn}
{\fontsize{11}{11pt}\text{and}}
\begin{fleqn}[1.5cm]
\begin{flalign*}
\sum \limits_{k=0}^{ \frac{p-7}{4} }\frac{1}{4k+3}&=\sum\limits_{j=\frac{p+3}{4}}^{\frac{p+1}{2}}\frac{1}{4( \frac{p-7}{4}-j+1)+3}=\sum\limits_{j=1}^{\frac{p-3}{4}}\frac{1}{p-4j},
\end{flalign*}
\end{fleqn}
we deduce the congruences
\begin{fleqn}[1.5cm]
\begin{flalign*}
\sum \limits_{k=0}^{ \frac{p-3}{4} }\frac{1}{4k+1}&\equiv\frac{1}{4}\sum \limits_{j=\frac{p+1}{4}}^{\frac{p-1}{2}}\frac{1}{j}
=\frac{1}{4}H_{[\frac{p}{2}]}-\frac{1}{4}H_{[\frac{p}{4}]}=\frac{1}{4}q_{p}(2)\,(\mathrm{mod}\hspace{0.3mm} p),\\
\sum \limits_{k=0}^{ \frac{p-3}{4} }\frac{1}{4k+2}&\equiv -\frac{1}{4}\sum \limits_{j=\frac{p+1}{4}}^{\frac{p-1}{2}}\frac{1}{j} =-\frac{1}{4}H_{[\frac{p}{2}]}+\frac{1}{4}H_{[\frac{p}{4}]}=-\frac{1}{4}q_{p}(2)\,(\mathrm{mod}\hspace{0.3mm} p)
\end{flalign*}
\end{fleqn}
{\fontsize{11}{11pt}\text{and}}
\begin{fleqn}[1.5cm]
\begin{flalign*}
\sum \limits_{k=0}^{ \frac{p-7}{4}}\frac{1}{4k+3}&\equiv-\frac{1}{4}\sum \limits_{j=1}^{\frac{p-3}{4}}\frac{1}{j} =-\frac{1}{4}H_{[\frac{p}{4}]}=\frac{3}{4}q_{p}(2)\,(\mathrm{mod}\hspace{0.3mm} p).
\end{flalign*}
\end{fleqn}

\end{proof}


\begin{lem}
{\sl Let $p\geq5$ be a prime number.

\noindent\textbf{\textit{1.}} If $p\equiv 1\,(\mathrm{mod}\hspace{0.3mm} 8)$ then
\begin{SMALL}
\begin{align}\label{equa:sum-p=1mod8}
&&\sum\limits_{k=0}^{\frac{p-1}{8}}\frac{1}{4k+1}\equiv2-\frac{1}{4}q_{p}(2)-\frac{1}{2}\chi_{p}\,(\mathrm{mod}\hspace{0.3mm} p),
&&\sum\limits_{k=0}^{\frac{p-1}{8}}\frac{1}{4k+2}\equiv \frac{2}{3}-\frac{1}{2}q_{p}(2)+\frac{1}{2}\chi_{p}\,(\mathrm{mod}\hspace{0.3mm} p),
&&\sum\limits_{k=0}^{\frac{p-1}{8}}\frac{1}{4k+3}\equiv\frac{2}{5}-\frac{1}{4}q_{p}(2)+\frac{1}{2}\chi_{p}\,(\mathrm{mod}\hspace{0.3mm} p).
\end{align}
\end{SMALL}

\noindent\textbf{\textit{2.}} If $p\equiv 3\,(\mathrm{mod}\hspace{0.3mm} 8)$ then
\begin{SMALL}
\begin{align}\label{equa:sum-p=3mod8}
&\sum\limits_{k=0}^{\frac{p-3}{8}}\frac{1}{4k+1}\equiv -\frac{1}{4}q_{p}(2)+\frac{1}{2}\chi_{p}\,(\mathrm{mod}\hspace{0.3mm} p),&&
\sum\limits_{k=0}^{\frac{p-3}{8}}\frac{1}{4k+2}\equiv 2-\frac{1}{2}q_{p}(2)+\frac{1}{2}\chi_{p}\,(\mathrm{mod}\hspace{0.3mm} p),&&
\sum\limits_{k=0}^{\frac{p-3}{8}}\frac{1}{4k+3}\equiv \frac{2}{3}-\frac{1}{4}q_{p}(2)-\frac{1}{2}\chi_{p}\,(\mathrm{mod}\hspace{0.3mm} p).
\end{align}
\end{SMALL}

\noindent\textbf{\textit{3.}} If $p\equiv 5\,(\mathrm{mod}\hspace{0.3mm} 8)$ then
\begin{SMALL}
\begin{align}\label{equa:sum-p=5mod8}
&\sum\limits_{k=0}^{\frac{p-5}{8}}\frac{1}{4k+1}\equiv -\frac{1}{4}q_{p}(2)-\frac{1}{2}\chi_{p}\,(\mathrm{mod}\hspace{0.3mm} p),&&
\sum\limits_{k=0}^{\frac{p-5}{8}}\frac{1}{4k+2}\equiv -\frac{1}{2}q_{p}(2)+\frac{1}{2}\chi_{p}\,(\mathrm{mod}\hspace{0.3mm} p),&&
\sum\limits_{k=0}^{\frac{p-5}{8}}\frac{1}{4k+3}\equiv 2-\frac{1}{4}q_{p}(2)+\frac{1}{2}\chi_{p}\,(\mathrm{mod}\hspace{0.3mm} p).
\end{align}
\end{SMALL}

\noindent\textbf{\textit{4.}} If $p\equiv 7\,(\mathrm{mod}\hspace{0.3mm} 8)$ then
\begin{SMALL}
\begin{align}\label{equa:sum-p=7mod8}
&\sum\limits_{k=0}^{\frac{p-7}{8}}\frac{1}{4k+1}\equiv -\frac{1}{4}q_{p}(2)+\frac{1}{2}\chi_{p}\,(\mathrm{mod}\hspace{0.3mm} p),&&
\sum\limits_{k=0}^{\frac{p-7}{8}}\frac{1}{4k+2}\equiv -\frac{1}{2}q_{p}(2)+\frac{1}{2}\chi_{p}\,(\mathrm{mod}\hspace{0.3mm} p),&&
\sum\limits_{k=0}^{\frac{p-7}{8}}\frac{1}{4k+3}\equiv -\frac{1}{4}q_{p}(2)-\frac{1}{2}\chi_{p}\,(\mathrm{mod}\hspace{0.3mm} p).
\end{align}
\end{SMALL}

}
\end{lem}


\begin{proof}
\textbf{\textit{i.}} For $p\equiv 1\,(\mathrm{mod}\hspace{0.3mm} 8),$ we have
\begin{fleqn}[1cm]
\begin{flalign*}
\sum \limits_{k=0}^{ \frac{p-1}{8} }\frac{1}{4k+1}=&\sum\limits_{j=\frac{p-1}{8}}^{\frac{p-1}{4}}\frac{1}{4( \frac{p-1}{4}-j)+1}=\sum\limits_{j=\frac{p-1}{8}}^{\frac{p-1}{4}}\frac{1}{p-4j},\\
\sum\limits_{k=0}^{\frac{p-1}{8}}\frac{1}{4k+2}=&\frac{2}{p+3}+\sum\limits_{j=1}^{\frac{p-1}{4}}\frac{1}{2j}-\sum\limits_{j=1}^{\frac{p-1}{8}}\frac{1}{4j}
\end{flalign*}
\end{fleqn}
{\fontsize{11}{11pt}\text{and}}
\begin{fleqn}[1cm]
\begin{flalign*}
\sum \limits_{k=0}^{ \frac{p-1}{8} }\frac{1}{4k+3}=&\sum\limits_{j=\frac{p-9}{8}}^{\frac{p-5}{4}}\frac{1}{4\big( \frac{p-5}{4}-j\big)+3}=\sum\limits_{j=\frac{p-1}{9}}^{\frac{p-5}{4}}\frac{1}{p-2-4j},\\
\end{flalign*}
\end{fleqn}
so that
\begin{fleqn}[1cm]
\begin{flalign*}
\sum \limits_{k=0}^{ \frac{p-1}{8} }\frac{1}{4k+1}\equiv& -\frac{1}{4}\sum \limits_{j=\frac{p-1}{8}}^{\frac{p-1}{4}}\frac{1}{j} =\frac{2}{p-1}-\frac{1}{4}\sum\limits_{j=1}^{\frac{p-1}{4}}\frac{1}{j}+\frac{1}{4}\sum\limits_{j=1}^{\frac{p-1}{8}}\frac{1}{j}\\
\equiv&2-\frac{1}{4}\sum\limits_{j=1}^{\frac{p-1}{4}}\frac{1}{j}+\frac{1}{4}\sum\limits_{j=1}^{\frac{p-1}{8}}\frac{1}{j}=2-\frac{1}{4}H_{[\frac{p}{2}]}+\frac{1}{4}H_{[\frac{p}{8}]}=2-\frac{1}{4}q_{p}(2)-\frac{1}{2}\chi_{p}\,(\mathrm{mod}\hspace{0.3mm} p),\\
\sum\limits_{k=0}^{\frac{p-1}{8}}\frac{1}{4k+2}\equiv&\frac{2}{3}+\frac{1}{2}\sum\limits_{j=1}^{\frac{p-1}{4}}\frac{1}{j}-\frac{1}{4}\sum\limits_{j=1}^{\frac{p-1}{8}}\frac{1}{j} =\frac{2}{3}+\frac{1}{2}H_{[\frac{p}{4}]}-\frac{1}{4}H_{[\frac{p}{8}]}=\frac{2}{3}-\frac{1}{2}q_{p}(2)+\frac{1}{2}\chi_{p}\,(\mathrm{mod}\hspace{0.3mm}p)
\end{flalign*}
\end{fleqn}
\hspace{0cm}{\fontsize{11}{11pt}\text{and, using the second congruence in~(\ref{equa:sum-p=1mod4}),}}
\begin{fleqn}[1cm]
\begin{align*}
\sum \limits_{k=0}^{ \frac{p-1}{8} }\frac{1}{4k+3}\equiv& -\sum \limits_{j=\frac{p-9}{8}}^{\frac{p-5}{4}}\frac{1}{-4j-2}
=-\sum\limits_{j=0}^{\frac{p-1}{4}}\frac{1}{4j+2}+\sum\limits_{j=0}^{\frac{p-1}{8}}\frac{1}{4j+2}+\frac{1}{p+1}-\frac{2}{p-5}-\frac{p+3}{2}\\
\equiv&-\frac{1}{4}\sum\limits_{j=0}^{\frac{p-1}{4}}\frac{1}{4j+2}+\sum\limits_{j=0}^{\frac{p-1}{8}}\frac{1}{4j+2}+\frac{11}{15}
=\frac{2}{5}+\frac{1}{4}H_{[\frac{p}{2}]}+\frac{1}{4}H_{[\frac{p}{4}]}-\frac{1}{4}H_{[\frac{p}{8}]}\\
&=\frac{2}{5}-\frac{1}{4}q_{p}(2)+\frac{1}{2}\chi_{p}\,(\mathrm{mod}\hspace{0.3mm}p).
\end{align*}
\end{fleqn}

\textbf{\textit{ii.}} Assume that $p\equiv 3\,(\mathrm{mod}\hspace{0.3mm} 8).$ From the equalities
\begin{fleqn}[1cm]
\begin{flalign*}
\sum \limits_{k=0}^{ \frac{p-3}{8} }\frac{1}{4k+1}=&\sum\limits_{j=\frac{p-3}{8}}^{\frac{p-3}{4}}\frac{1}{4( \frac{p-7}{4}-j+1)+1}=\sum\limits_{j=\frac{p-3}{8}}^{\frac{p-3}{4}}\frac{1}{p-4j-2},\\
\sum\limits_{k=0}^{\frac{p-3}{8}}\frac{1}{4k+2}=&\frac{2}{p+1}+\frac{1}{2}\sum\limits_{j=1}^{\frac{p-3}{4}}\frac{1}{j}-\frac{1}{4}\sum\limits_{j=1}^{\frac{p-3}{8}}\frac{1}{j}
\end{flalign*}
\end{fleqn}
\hspace{0cm}{\fontsize{11}{11pt}\text{and}}
\begin{fleqn}[1cm]
\begin{flalign*}
\sum \limits_{k=0}^{ \frac{p-3}{8} }\frac{1}{4k+3}=&\sum\limits_{j=\frac{p-3}{8}}^{\frac{p-3}{4}}\frac{1}{4\big( \frac{p-3}{4}-j\big)+3}=\sum\limits_{j=\frac{p-3}{8}}^{\frac{p-3}{4}}\frac{1}{p-4j},\\
\end{flalign*}
\end{fleqn}
we obtain the congruences
\begin{fleqn}[1cm]
\begin{flalign*}
\sum \limits_{k=0}^{ \frac{p-3}{8} }\frac{1}{4k+3}\equiv& -\frac{1}{4}\sum \limits_{j=\frac{p-3}{8}}^{\frac{p-3}{4}}\frac{1}{j} =\frac{2}{p-3}-\frac{1}{4}\sum\limits_{j=1}^{\frac{p-3}{4}}\frac{1}{j}+\frac{1}{4}\sum\limits_{j=1}^{\frac{p-3}{8}}\frac{1}{j}\\
\equiv&\frac{2}{3}-\frac{1}{4}\sum\limits_{j=1}^{\frac{p-1}{4}}\frac{1}{j}+\sum\limits_{j=1}^{\frac{p-1}{8}}\frac{1}{j}
=\frac{2}{3}-\frac{1}{4}H_{[\frac{p}{4}]}+\frac{1}{4}H_{[\frac{p}{8}]}=\frac{2}{3}-\frac{1}{4}q_{p}(2)-\frac{1}{2}\chi_{p}\,(\mathrm{mod}\hspace{0.3mm}p)
,\\
\sum\limits_{k=0}^{\frac{p-3}{8}}\frac{1}{4k+2}\equiv&2+\frac{1}{2}\sum\limits_{j=1}^{\frac{p-3}{4}}\frac{1}{j}-\frac{1}{4}\sum\limits_{j=1}^{\frac{p-3}{8}}\frac{1}{j} =2+\frac{1}{2}H_{[\frac{p}{4}]}-\frac{1}{4}H_{[\frac{p}{8}]}=2-\frac{1}{2}q_{p}(2)+\frac{1}{2}\chi_{p}\,(\mathrm{mod}\hspace{0.3mm}p)
\end{flalign*}
\end{fleqn}
\hspace{0cm}{\fontsize{11}{11pt}\text{and, using the second congruence~in~(\ref{equa:sum-p=3mod4}),}}
\begin{fleqn}[1cm]
\begin{align*}
\sum \limits_{k=0}^{ \frac{p-3}{8} }\frac{1}{4k+1}\equiv& -\sum \limits_{j=\frac{p-3}{8}}^{\frac{p-3}{4}}\frac{1}{4j+2} =-\frac{2}{p+1}+\sum\limits_{j=0}^{\frac{p-3}{8}}\frac{1}{4j+2}-\sum\limits_{j=0}^{\frac{p-3}{4}}\frac{1}{4j+2}\\
\equiv&-2+\sum\limits_{j=0}^{\frac{p-3}{8}}\frac{1}{4j+2}-\sum\limits_{j=0}^{\frac{p-3}{4}}\frac{1}{4j+2}=\frac{1}{4}H_{[\frac{p}{2}]}+\frac{1}{4}H_{[\frac{p}{4}]}-\frac{1}{4}H_{[\frac{p}{8}]}=2-\frac{1}{4}q_{p}(2)-\frac{1}{2}\chi_{p}\,(\mathrm{mod}\hspace{0.3mm} p).
\end{align*}
\end{fleqn}

\textbf{\textit{iii.}} Consider the case where $p\equiv 5\,(\mathrm{mod}\hspace{0.3mm} 8).$ We have
\begin{fleqn}[1cm]
\begin{flalign*}
\sum \limits_{k=0}^{ \frac{p-5}{8} }\frac{1}{4k+1}=&\sum\limits_{j=\frac{p+3}{8}}^{\frac{p-1}{4}}\frac{1}{4( \frac{p-5}{4}-j+1)+1}=\sum\limits_{j=\frac{p+3}{8}}^{\frac{p-1}{4}}\frac{1}{p-4j},\\
\sum\limits_{k=0}^{\frac{p-5}{8}}\frac{1}{4k+2}=&\frac{2}{p-1}+\frac{1}{2}\sum\limits_{j=1}^{\frac{p-5}{4}}\frac{1}{j}-\frac{1}{4}\sum\limits_{j=1}^{\frac{p-5}{8}}\frac{1}{j}
\end{flalign*}
\end{fleqn}
\hspace{0cm}{\fontsize{11}{11pt}\text{and}}
\begin{fleqn}[1cm]
\begin{flalign*}
\sum \limits_{k=0}^{ \frac{p-5}{8} }\frac{1}{4k+3}=&\sum\limits_{j=\frac{p-5}{8}}^{\frac{p-5}{4}}\frac{1}{4\big( \frac{p-1}{4}-j-1\big)+3}=\sum\limits_{j=\frac{p-5}{8}}^{\frac{p-5}{4}}\frac{1}{p-4j-2}.\\
\end{flalign*}
\end{fleqn}
It follows that
\begin{fleqn}[1cm]
\begin{flalign*}
\sum \limits_{k=0}^{ \frac{p-5}{8} }\frac{1}{4k+1}\equiv& -\frac{1}{4}\sum \limits_{j=\frac{p+3}{8}}^{\frac{p-1}{4}}\frac{1}{j} =-\frac{1}{4}H_{[\frac{p}{4}]}+\frac{1}{4}H_{[\frac{p}{8}]}=2-\frac{1}{4}q_{p}(2)-\frac{1}{2}\chi_{p}\,(\mathrm{mod}\hspace{0.3mm} p),\\
\sum\limits_{k=0}^{\frac{p-5}{8}}\frac{1}{4k+2}\equiv&\frac{1}{2}\sum\limits_{j=1}^{\frac{p-1}{4}}\frac{1}{j}-\frac{1}{4}\sum\limits_{j=1}^{\frac{p-5}{8}}\frac{1}{j} =\frac{1}{2}H_{[\frac{p}{4}]}-\frac{1}{4}H_{[\frac{p}{8}]}=-\frac{1}{2}q_{p}(2)+\frac{1}{2}\chi_{p}\,(\mathrm{mod}\hspace{0.3mm}p)
\end{flalign*}
\end{fleqn}
\hspace{0cm}{\fontsize{11}{11pt}\text{and}}
\begin{fleqn}[1cm]
\begin{align*}
\sum \limits_{k=0}^{ \frac{p-5}{8} }\frac{1}{4k+3}\equiv& \sum \limits_{j=0}^{\frac{p-13}{8}}\frac{1}{4j+2}-\sum \limits_{j=0}^{\frac{p-5}{4}}\frac{1}{4j+2} =\sum\limits_{j=0}^{\frac{p-5}{8}}\frac{1}{4j+2}-\sum\limits_{j=0}^{\frac{p-1}{4}}\frac{1}{4j+2}+\frac{1}{p+1}-\frac{2}{p-1}\\
\equiv&3+\sum\limits_{j=0}^{\frac{p-5}{8}}\frac{1}{4j+2}-\sum\limits_{j=0}^{\frac{p-1}{4}}\frac{1}{4j+2}\\
\equiv& 2+\frac{1}{4}H_{[\frac{p}{2}]}+\frac{1}{4}H_{[\frac{p}{4}]}-\frac{1}{4}H_{[\frac{p}{8}]}=2-\frac{1}{4}q_{p}(2)+\frac{1}{2}\chi_{p}\,(\mathrm{mod}\hspace{0.3mm}p),
\end{align*}
\end{fleqn}

\textbf{\textit{iv.}} For $p\equiv 7\,(\mathrm{mod}\hspace{0.3mm} 8)$, we have
\begin{fleqn}[1cm]
\begin{flalign*}
\sum \limits_{k=0}^{ \frac{p-7}{8} }\frac{1}{4k+1}=&\sum\limits_{j=\frac{p+1}{8}}^{\frac{p-3}{4}}\frac{1}{4( \frac{p-3}{4}-j)+1}=\sum\limits_{j=\frac{p+1}{8}}^{\frac{p-3}{4}}\frac{1}{p-4j-2},\\
\sum\limits_{k=0}^{\frac{p-7}{8}}\frac{1}{4k+2}=&\frac{2}{p-3}+\frac{1}{2}\sum\limits_{j=1}^{\frac{p-7}{4}}\frac{1}{j}-\frac{1}{4}\sum\limits_{j=1}^{\frac{p-7}{8}}\frac{1}{j}=\frac{1}{2}\sum\limits_{j=1}^{\frac{p-3}{4}}\frac{1}{j}-\frac{1}{4}\sum\limits_{j=1}^{\frac{p-7}{8}}\frac{1}{j}
\end{flalign*}
\end{fleqn}
\hspace{0cm}{\fontsize{11}{11pt}\text{and}}
\begin{fleqn}[1cm]
\begin{flalign*}
\sum \limits_{k=0}^{ \frac{p-7}{8} }\frac{1}{4k+3}=&\sum\limits_{j=\frac{p+1}{8}}^{\frac{p-3}{4}}\frac{1}{4\big( \frac{p-7}{4}-j+1\big)+3}=\sum\limits_{j=\frac{p+1}{8}}^{\frac{p-3}{4}}\frac{1}{p-4j}.\\
\end{flalign*}
\end{fleqn}
Thus
\begin{fleqn}[1cm]
\begin{flalign*}
\sum \limits_{k=0}^{ \frac{p-7}{8} }\frac{1}{4k+3}\equiv& -\frac{1}{4}\sum \limits_{j=\frac{p+1}{8}}^{\frac{p-3}{4}}\frac{1}{j} =-\frac{1}{4}H_{[\frac{p}{4}]}+\frac{1}{4}H_{[\frac{p}{8}]}=-\frac{1}{4}q_{p}(2)-\frac{1}{2}\chi_{p}\,(\mathrm{mod}\hspace{0.3mm}p)),\\
\sum\limits_{k=0}^{\frac{p-5}{8}}\frac{1}{4k+2}\equiv&\frac{1}{2}\sum\limits_{j=1}^{\frac{p-1}{4}}\frac{1}{j}-\frac{1}{4}\sum\limits_{j=1}^{\frac{p-5}{8}}\frac{1}{j} =\frac{1}{2}H_{[\frac{p}{4}]}-\frac{1}{4}H_{[\frac{p}{8}]}=-\frac{1}{2}q_{p}(2)+\frac{1}{2}\chi_{p}\,(\mathrm{mod}\hspace{0.3mm}p)
\end{flalign*}
\end{fleqn}
\hspace{0cm}{\fontsize{11}{11pt}\text{and}}
\begin{fleqn}[1cm]
\begin{align*}
\sum \limits_{k=0}^{ \frac{p-7}{8} }\frac{1}{4k+1}\equiv& -\sum \limits_{j=\frac{p+1}{8}}^{\frac{p-3}{4}}\frac{1}{4j+2} =\frac{1}{4}H_{[\frac{p}{2}]}+\frac{1}{4}H_{[\frac{p}{4}]}-\frac{1}{4}H_{[\frac{p}{8}]}=-\frac{1}{4}q_{p}(2)+\frac{1}{2}\chi_{p}\,(\mathrm{mod}\hspace{0.3mm} p).
\end{align*}
\end{fleqn}
\end{proof}


\begin{pro}\label{P}
{\sl Let $p\geq5$ be a prime number and $n,k$ be positive integers. We have
\begin{align}
\binom{np-1}{4k}_{3}&\equiv1-np\left(\frac{3}{4}H_{k}+\sum \limits_{j=0}^{k-1}\frac{1}{4j+3}\right)\,(\mathrm{mod}\hspace{0.3mm}p^{2}),&4k&\leq p-1,\label{Prop2_np-1_4k}\\
\binom{np-1}{4k+1}_{3}&\equiv-1+np\left(\frac{3}{4}H_{k}+\sum \limits_{j=0}^{k}\frac{1}{4j+1}\right)\,(\mathrm{mod}\hspace{0.3mm}p^{2}),&4k&+1\leq p-1,\label{Prop2_np-1_4k+1}\\
\binom{np-1}{4k+2}_{3}&\equiv np\left(\sum\limits_{j=0}^{k}\frac{1}{4j+2}-\sum \limits_{j=0}^{k}\frac{1}{4j+1}\right)\,(\mathrm{mod}\hspace{0.3mm} p^2),&4k&+2\leq p-1,\label{Prop2_np-1_4k+2}\\
\binom{np-1}{4k+3}_{3}&\equiv np\left(\sum\limits_{j=0}^{k}\frac{1}{4j+3}-\sum \limits_{j=0}^{k}\frac{1}{4j+2}\right)\,(\mathrm{mod}\hspace{0.3mm} p^2)
,&4k&+2\leq p-1.\label{Prop2_np-1_4k+3}
\end{align}
}
\end{pro}



\begin{proof}
We have
\[
(1+x+x^{2}+x^3)^{n}=(1+x^2)^n(1+x)^n=\left(\sum\limits_{j=0}^{n}\binom{n}{j}_{3}x^{2j}\right)\left(\sum\limits_{l=0}^{n}\binom{n}{l}_{3}x^{l}\right).
\]
It follows that
\begin{equation}\label{Forme_explicite bi3nomial}
\binom{n}{k}_{3}=\underset{\underset{0\leq l,j\leq n}{2j+l=k}}{\sum\limits}\binom{n}{j}\binom{n}{l}=\sum\limits_{j=0}^{\min (n,[\frac{k}{2}])}
\binom{n}{j}\binom{n}{k-2j}.
\end{equation}
Combining (\ref{Forme_explicite bi3nomial}) with the relation $\binom{np-1}{k}=(-1)^{k}\overset{k}{\underset{j=1}{\prod}}\left(1-\frac{np}{j}\right)\equiv (-1)^{k}\big(1-npH_{k}\big)\hspace{1mm}(\mathrm{mod}\hspace{0.3mm} p^2)$ we obtain that
\begin{align*}
\binom{np-1}{k}_{3}&=\sum \limits_{j=0}^{\min (np-1,[\frac{k}{2}])}\binom{np-1}{j}\binom{np-1}{k-2j}\\
&\equiv  \sum \limits_{j=0}^{\min (np-1,[\frac{k}{2}])}(-1)^{k-j}\Big(1-np(H_{j}+H_{k-2j})\Big)\hspace{1mm}(\mathrm{mod}\hspace{0.3mm} p^2).
\end{align*}
\newpage
\hfill
\vspace{0.4cm}

\noindent Then, for the congruence (\ref{Prop2_np-1_4k}), we have
\begin{flushleft}
\begin{align*}
\hspace{-0.84cm}
\binom{np-1}{4k}_{3}\equiv & \sum \limits_{j=0}^{2k}(-1)^{j}\Big(1-np(H_{j}+H_{k-2j})\Big)\hspace{1mm}(\mathrm{mod}\hspace{0.3mm} p^2)\\
=&1-np\left(\sum \limits_{j=0}^{2k}(-1)^{j}H_{j}+\sum \limits_{j=0}^{2k}(-1)^{j}H_{4k-2j}\right)\\
=&1-np\sum \limits_{j=0}^{2k}(-1)^{j}\big(H_{j}+H_{2j}\big)\\
=&1-np\left(H_{2k}+H_{4k}-\sum\limits_{j=0}^{k-1}\Big(H_{2j+1}-H_{2j}+H_{4j+2}-H_{4j}\Big)\right) \\
=&1-np\left[\sum\limits_{j=1}^{k}\frac{1}{2j}+\sum \limits_{j=0}^{k-1}\frac{1}{2j+1}+\sum \limits_{j=1}^{k}\frac{1}{4j}+\sum \limits_{j=0}^{k-1}\left(\frac{1}{4j+1}+\frac{1}{4j+2}+\frac{1}{4j+3}\right)\right]\\
&\hspace{2mm}+np\sum\limits_{j=0}^{k-1}\left(\frac{1}{2j+1}+\frac{1}{4j+1}+\frac{1}{4j+2}\right)  \\
=&1-np\left(\frac{3}{4}H_{k}+\sum \limits_{j=0}^{k-1}\frac{1}{4j+3}\right)\hspace{1mm}(\mathrm{mod}\hspace{0.3mm} p^2)  .
\end{align*}
\end{flushleft}
\medskip

\noindent Now, we show the congruence (\ref{Prop2_np-1_4k+1}). We have
\begin{flushleft}
\begin{align*}
\hspace{-1.5cm}
\binom{np-1}{4k+1}_{3}\equiv &-\sum \limits_{j=0}^{2k}(-1)^{j}\Big(1-np(H_{j}+H_{4k+1-2j})\Big)\left(\mathrm{mod}\hspace{0.3mm}p^{2}\right)\\
=&-1+np \sum \limits_{j=0}^{2k}(-1)^{j}\big(H_{j}+H_{2j+1}\big)\\
=&-1+np\left(H_{2k}+H_{4k+1}-\sum \limits_{j=0}^{k-1}\Big(H_{2j+1}-H_{2j}+H_{4j+3}-H_{4j+1}\Big)\right)\\
=&-1+np\left[ \sum \limits_{j=1}^{k}\frac{1}{2j}+\sum \limits_{j=0}^{k-1}\frac{1}{2j+1}+\frac{1}{4k+1}+\sum \limits_{j=1}^{k}\frac{1}{4j}+\sum \limits_{j=0}^{k-1}\left(\frac{1}{4j+1}+\frac{1}{4j+2}+\frac{1}{4j+3}\right)\right]\\
&\hspace{6mm}-np\sum \limits_{j=0}^{k-1}\left(\frac{1}{2j+1}+\frac{1}{4j+2}+\frac{1}{4j+3}\right)  \\
=&-1+np\left(\frac{3}{4}H_{k}+\sum \limits_{j=0}^{k}\frac{1}{4j+1}\right)\,(\mathrm{mod}\hspace{0.3mm}p^{2}).
\end{align*}
\end{flushleft}
Next, we prove the congruence (\ref{Prop2_np-1_4k+2}). We have
\begin{align*}
\binom{np-1}{4k+2}_{3}  & \equiv \sum \limits_{j=0}^{2k+1}(-1)^{j}\Big(1-np(H_{j}+H_{4k+2-2j})\Big)\,(\mathrm{mod}\hspace{0.3mm}p^{2}).\\
& =np \sum \limits_{j=0}^{2k+1}(-1)^{j}\big(H_{2j}-H_{j}\big)\\
& =np\left(H_{2k+1}-H_{2k}+H_{4k}-H_{4k+2}+\sum \limits_{j=0}^{k-1}\frac{1}{4j+2}-\sum \limits_{j=0}^{k-1}\frac{1}{4j+1}\right)\\
& =np\left(\frac{1}{2k+1}-\frac{1}{4k+1}-\frac{1}{4k+2}+\sum \limits_{j=0}^{k-1}\frac{1}{4j+2}-\sum \limits_{j=0}^{k-1}\frac{1}{4j+1}\right)\\
&=np\left(\sum \limits_{j=0}^{k}\frac{1}{4j+2}-\sum \limits_{j=0}^{k}\frac{1}{4j+1}\right)\,(\mathrm{mod}\hspace{0.3mm}p^{2}).
\end{align*}
\noindent Finally, for the congruence (\ref{Prop2_np-1_4k+3}), we have
\begin{align*}
\binom{np-1}{4k+3}_{3}& \equiv-\sum \limits_{j=0}^{2k+1}(-1)^{j}\Bigg(1-np\big(H_{j}+H_{4k+3-2j}\big)\Bigg)\\
& =np\left( \sum \limits_{j=0}^{2k+1}(-1)^{j}\big(H_{j}-H_{2j+1}\big)\right)\\
& =np\left(H_{2k}-H_{2k+1}+H_{4k+3}-H_{4k+1}+\sum \limits_{j=0}^{k-1}\frac{1}{4j+3}-\sum \limits_{j=0}^{k-1}\frac{1}{4j+2}\right)\\
& =np\left(\frac{1}{4k+3}-\frac{1}{4k+2}+\sum \limits_{j=0}^{k-1}\frac{1}{4j+3}-\sum \limits_{j=0}^{k-1}\frac{1}{4j+2}\right)\\
& =np\left(\sum \limits_{j=0}^{k}\frac{1}{4j+3}-\sum \limits_{j=0}^{k}\frac{1}{4j+2}\right)\,(\mathrm{mod}\hspace{0.3mm}p^{2}).
\end{align*}

\end{proof}


\section{Proof of the main results}
\bigskip

\begin{proof}[\sl Proof of Theorem \ref{thmalph:Bi3nomial-np-1-p-1-p-1-sur-2}]
For $p\equiv1$ $\left(  \mathrm{mod}\hspace{0.3mm} 4\right)$ let $4k=p-1$ in the congruence (\ref{Prop2_np-1_4k}). Then, by the congruences
(\ref{Harmonic4}) and (\ref{equa:sum-p=1mod4}) we get
\[
\binom{np-1}{p-1}_{3}\equiv1-np\left(\frac{3}{4}H_{\frac{p-1}{4}}+\sum \limits_{j=0}^{\frac{p-5}{4}}
\frac{1}{4j+3}\right)\equiv1+2npq_{p}(2)\,(\mathrm{mod}\hspace{0.3mm}p^{2}).
\]
For $p\equiv3$ $\left(  \mathrm{mod}\hspace{0.3mm} 4\right)  $ let $4k+2=p-1$ in the congruence (\ref{Prop2_np-1_4k+2}). Then, by the congruence (\ref{equa:sum-p=3mod4}) we get
\[
\binom{np-1}{p-1}_{3}\equiv np\left(\sum \limits_{j=0}^{\frac{p-3}{4}}\frac{1}{4j+2}-\sum \limits_{j=0}^{\frac{p-3}{4}}\frac{1}{4j+1}\right)  =-\frac{1}{2}npq_{p}(2)\hspace{1mm}(\mathrm{mod}\hspace{0.3mm} p^2).
\]
For $p\equiv1$ $\left(  \mathrm{mod}\hspace{0.3mm}8\right)  $ let $4k=\frac{p-1}{2}$ in the congruence (\ref{Prop2_np-1_4k}). Then, by the congruences
(\ref{Harmonic8}) and (\ref{equa:sum-p=1mod8}) we get
\[
\binom{np-1}{\frac{p-1}{2}}_{3}\equiv1-np\left(\frac{3}{4}H_{\frac{p-1}{8}}+
\sum \limits_{j=0}^{\frac{p-9}{8}}
\frac{1}{4j+3}\right)  \equiv1+np\Big(\frac{13}{4}q_{p}(2)+\chi_{p}\Big)\hspace{1mm}(\mathrm{mod}\hspace{0.3mm} p^2).
\]
For $p\equiv3$ $\left(  \mathrm{mod}\hspace{0.3mm}8\right)  $ let $4k+1=\frac{p-1}{2}$ in the congruence (\ref{Prop2_np-1_4k+1}). Then, by the congruences
(\ref{Harmonic8}) and (\ref{equa:sum-p=3mod8}) we get
\[
\binom{np-1}{\frac{p-1}{2}}_{3}\equiv-1+np\left( \frac{3}{4}H_{\frac{p-3}{8}}+\sum \limits_{j=0}^{\frac{p-3}{8}}
\frac{1}{4j+1}\right)=-1-np\left(\frac{13}{4}q_{p}(2)+\chi_{p}\right)\hspace{1mm}(\mathrm{mod}\hspace{0.3mm} p^2).
\]

\noindent For $p\equiv5$ $\left(  \mathrm{mod}\hspace{0.3mm}8\right)  $ let $4k+2=\frac{p-1}{2}$ in the congruence (\ref{Prop2_np-1_4k+2}). Then, by the congruence
(\ref{equa:sum-p=5mod8}) we get
\[
\binom{np-1}{\frac{p-1}{2}}_{3}\equiv np\left(\sum \limits_{j=0}^{\frac{p-5}{8}}\frac{1}{4j+2}-\sum \limits_{j=0}^{\frac{p-5}{8}}\frac{1}{4j+1}
\right)=-1-np\left(\frac{13}{4}q_{p}(2)+\chi_{p}\right)\hspace{1mm}(\mathrm{mod}\hspace{0.3mm} p^2).
\]

\noindent For $p\equiv7$ $\left(  \mathrm{mod}\hspace{0.3mm}8\right)  $ let $4k+3=\frac{p-1}{2}$ in the congruence (\ref{Prop2_np-1_4k+3}). Then, by the congruence
(\ref{equa:sum-p=7mod8}) we get
\[
\binom{np-1}{\frac{p-1}{2}}_{3}\equiv np\left(\sum \limits_{j=0}^{\frac{p-7}{8}}\frac{1}{4j+3}-\sum \limits_{j=0}^{\frac{p-7}{8}}\frac{1}{4j+2}
\right)=np\left(\frac{1}{4}q_{p}(2)-\chi_{p}\right)\hspace{1mm}(\mathrm{mod}\hspace{0.3mm} p^2).
\]
\end{proof}
\begin{proof}[\sl Proof of Proposition \ref{propalph:sommes-Bi3nomial-np-1-k}]For $k\in \left \{  0,1,\ldots,\left[\frac{p}{4}\right]  -1\right \} ,$ from Proposition~\ref{P} we deduce that
\begin{equation}\label{sommes-quadrinomials-4k--4k+3}
\binom{np-1}{4k}_{3}+\binom{np-1}{4k+1}_{3}+\binom{np-1}{4k+2}_{3}+\binom{np-1}{4k+3}_{3}\equiv\frac{np}{4k+3}\hspace{1mm}(\mathrm{mod}\hspace{0.3mm} p^2).
\end{equation}
To show the congruences (\ref{equa:sommes-Bi3nomial-np-1-k-k-0-p-1}), let
\[
\sum \limits_{k=0}^{p-1}\binom{np-1}{k}_{3}=\sum \limits_{k=0}^{[\frac{p-1}{4}]}\binom{np-1}{4k}_{3}+\sum \limits_{k=0}^{[\frac{p-2}{4}]}\binom{np-1}{4k+1}_{3}+\sum \limits_{k=0}^{[\frac{p-3}{4}]}\binom{np-1}{4k+2}_{3}+\sum \limits_{k=0}^{[\frac{p-4}{4}]}\binom{np-1}{4k+3}_{3}.
\]

\noindent For $p\equiv1$ $\left(\mathrm{mod}\hspace{0.3mm}3\right),$ by the congruences (\ref{sommes-quadrinomials-4k--4k+3}), (\ref{equa:congruence-Bi3nomial-np-1-p-1}) and (\ref{equa:sum-p=1mod4}), we obtain
\begin{align*}
\sum \limits_{k=0}^{p-1}\binom{np-1}{k}_{3}&=\sum \limits_{k=0}^{\frac{p-1}{4}}\binom{np-1}{4k}_{3}+\sum \limits_{k=0}^{\frac{p-1}{4}-1}\binom{np-1}{4k+1}_{3}+\sum \limits_{k=0}^{\frac{p-1}{4}-1}\binom{np-1}{4k+2}_{3}+\sum \limits_{k=0}^{\frac{p-1}{4}-1}\binom{np-1}{4k+3}_{3}\\
&=\binom{np-1}{p-1}_{3}+\sum \limits_{k=0}^{\frac{p-5}{4}}\left\{\binom{np-1}{4k+1}_{3}+\binom{np-1}{4k+2}_{3}+\binom{np-1}{4k+3}_{3}\right\}\\
&\equiv 1+2npq_{p}(2)+np\sum \limits_{k=0}^{\frac{p-5}{4}}\frac{1}{4k+3}\\
&\equiv 1+\frac{9}{4}npq_{p}(2).
\end{align*}

\noindent For $p\equiv3$ $\left(\mathrm{mod}\hspace{0.3mm}4\right),$ by the congruences (\ref{sommes-quadrinomials-4k--4k+3}), (\ref{Prop2_np-1_4k+3}) and (\ref{equa:sum-p=3mod4}), we obtain
\begin{align*}
\sum \limits_{k=0}^{p-1}\binom{np-1}{k}_{3}&=\sum \limits_{k=0}^{\frac{p-3}{4}}\binom{np-1}{4k}_{3}+\sum \limits_{k=0}^{\frac{p-3}{4}}\binom{np-1}{4k+1}_{3}+\sum \limits_{k=0}^{\frac{p-3}{4}}\binom{np-1}{4k+2}_{3}+\sum \limits_{k=0}^{\frac{p-7}{4}}\binom{np-1}{4k+3}_{3}\\
&=-\binom{np-1}{p}_{3}+\sum \limits_{k=0}^{\frac{p-3}{4}}\left\{\binom{np-1}{4k}_{3}+\binom{np-1}{4k+1}_{3}+\binom{np-1}{4k+2}_{3}+\binom{np-1}{4k+3}_{3}\right\}\\
&= np\sum \limits_{k=0}^{\frac{p-3}{4}}\frac{1}{4k+2}\\
&=-\frac{1}{4}q_{p}(2).
\end{align*}
To show the congruences (\ref{equa:sommes-Bi3nomial-np-1-k-k-0-p-1-sur-2}), let
\[
\sum \limits_{k=0}^{\frac{p-1}{2}}\binom{np-1}{k}_{3}=\sum \limits_{k=0}^{[\frac{p-1}{8}]}\binom{np-1}{4k}_{3}+\sum \limits_{k=0}^{[\frac{p-2}{8}]}\binom{np-1}{4k+1}_{3}+\sum \limits_{k=0}^{[\frac{p-4}{8}]}\binom{np-1}{4k+2}_{3}+\sum \limits_{k=0}^{[\frac{p-6}{8}]}\binom{np-1}{4k+3}_{3}.
\]

\noindent For $p\equiv1$ $\left(  \mathrm{mod}\hspace{0.3mm}8\right),$ by the congruences (\ref{sommes-quadrinomials-4k--4k+3}), (\ref{equa:congruence-Bi3nomial-np-1-p-1-sur-2}) and (\ref{equa:sum-p=7mod8}), we get
\begin{align*}
\sum \limits_{k=0}^{\frac{p-1}{2}}\binom{np-1}{k}_{3}&=\sum\limits_{k=0}^{\frac{p-1}{8}}\binom{np-1}{4k}_{3}+\sum \limits_{k=0}^{\frac{p-9}{8}}\binom{np-1}{4k+1}_{3}+\sum \limits_{k=0}^{\frac{p-9}{8}}\binom{np-1}{4k+2}_{3}+\sum \limits_{k=0}^{\frac{p-9}{8}}\binom{np-1}{4k+3}_{3}\\
&=\binom{np-1}{\frac{p-1}{2}}_{3}+\sum\limits_{k=0}^{\frac{p-9}{8}}\left\{\binom{np-1}{4k}_{3}+\binom{np-1}{4k+1}_{3}+\binom{np-1}{4k+2}_{3}+\binom{np-1}{4k+3}_{3}\right\}\\
&=\binom{np-1}{\frac{p-1}{2}}_{3}+np\left(\sum \limits_{k=0}^{\frac{p-9}{8}}\frac{1}{4k+3}\right)\\
&\equiv 1+np\left(\frac{13}{4}q_{p}(2)+\chi(p)-\frac{2}{5}+\sum \limits_{k=0}^{\frac{p-1}{8}}\frac{1}{4k+3}\right)\\
&=1+\frac{3}{2}np\Big(2q_{p}(2)+\chi(p)\Big).
\end{align*}
For $p\equiv3$ $\left(  \mathrm{mod}\hspace{0.3mm}8\right)  ,$ by the congruences (\ref{sommes-quadrinomials-4k--4k+3}), (\ref{equa:congruence-Bi3nomial-np-1-p-1-sur-2}) and (\ref{equa:sum-p=7mod8}), we get
\begin{align*}
\sum \limits_{k=0}^{\frac{p-1}{2}}\binom{np-1}{k}_{3}&=\sum\limits_{k=0}^{\frac{p-3}{8}}\binom{np-1}{4k}_{3}+\sum \limits_{k=0}^{\frac{p-3}{8}}\binom{np-1}{4k+1}_{3}+\sum \limits_{k=0}^{\frac{p-11}{8}}\binom{np-1}{4k+2}_{3}+\sum\limits_{k=0}^{\frac{p-11}{8}}\binom{np-1}{4k+3}_{3}\\
&=-\binom{np-1}{\frac{p+1}{2}}_{3}-\binom{np-1}{\frac{p+3}{2}}_{3}+np\sum\limits_{k=0}^{\frac{p-3}{8}}\frac{1}{4k+3}\\
&=np\sum\limits_{k=0}^{\frac{p-3}{8}}\frac{1}{4k+1}\\
&=-\frac{1}{4}np\Big(q_{p}(2)-2\chi(p)\Big).
\end{align*}
For $p\equiv5$ $\left(  \mathrm{mod}\hspace{0.3mm}8\right) ,$ by the congruences (\ref{sommes-quadrinomials-4k--4k+3}) and (\ref{equa:sum-p=7mod8}), we get

\begin{align*}
\sum \limits_{k=0}^{\frac{p-1}{2}}\binom{np-1}{k}_{3}&=\sum\limits_{k=0}^{\frac{p-5}{8}}\binom{np-1}{4k}_{3}+\sum \limits_{k=0}^{\frac{p-5}{8}}\binom{np-1}{4k+1}_{3}+\sum \limits_{k=0}^{\frac{p-5}{8}}\binom{np-1}{4k+2}_{3}+\sum\limits_{k=0}^{\frac{p-13}{8}}\binom{np-1}{4k+3}_{3}\\
&=-\binom{np-1}{\frac{p+1}{3}}_{3}+\sum\limits_{k=0}^{\frac{p-5}{8}}\left\{\binom{np-1}{4k}_{3}+\binom{np-1}{4k+1}_{3}+\binom{np-1}{4k+2}_{3}+\binom{np-1}{4k+3}_{3}\right\}\\
&=-np\sum \limits_{k=0}^{\frac{p-5}{8}}\frac{1}{4k+3}\\
&=-\frac{1}{2}np\Big(q_{p}(2)-\chi(p)\Big).
\end{align*}

\noindent For $p\equiv7$ $\left(  \mathrm{mod}\hspace{0.3mm}8\right)  ,$ by the congruences (\ref{sommes-quadrinomials-4k--4k+3}) and (\ref{equa:sum-p=7mod8}), we get

\begin{align*}
\sum \limits_{k=0}^{\frac{p-1}{2}}\binom{np-1}{k}_{3}&=\sum\limits_{k=0}^{\frac{p-7}{8}}\binom{np-1}{4k}_{3}+\sum \limits_{k=0}^{\frac{p-7}{8}}\binom{np-1}{4k+1}_{3}+\sum \limits_{k=0}^{\frac{p-7}{8}}\binom{np-1}{4k+2}_{3}+\sum\limits_{k=0}^{\frac{p-7}{8}}\binom{np-1}{4k+3}_{3}\\
&=\sum\limits_{k=0}^{\frac{p-7}{8}}\left\{\binom{np-1}{4k}_{3}+\binom{np-1}{4k+1}_{3}+\binom{np-1}{4k+2}_{3}+\binom{np-1}{4k+3}_{3}\right\}\\
&=np\sum \limits_{k=0}^{\frac{p-7}{8}}\frac{1}{4k+3}\\
&=-\frac{1}{4}np\Big(q_{p}(2)+2\chi(p)\Big).
\end{align*}
\end{proof}

\begin{proof}[\sl Proof of Corollary \ref{coralph:Bi3nomial-np2-1-k}]
This follows without difficulty from Proposition \ref{P}.
\end{proof}


\end{document}